\documentclass[reqno,11pt]{amsart}
\usepackage{amssymb}
\usepackage{amsmath}
\usepackage{stmaryrd}
\usepackage{graphicx}
\usepackage{color}
\usepackage{cite}
\usepackage[font=small]{caption}
\usepackage{bm}

\usepackage{tikz}
\usepackage{pgfplots}
\pgfplotsset{compat=1.10}
\usepgfplotslibrary{fillbetween}
\usetikzlibrary{patterns}
\newtheorem{theorem}{Theorem}[section]
\newtheorem{corollary}[theorem]{Corollary}
\newtheorem{lemma}[theorem]{Lemma}
\newtheorem{proposition}[theorem]{Proposition}

\newcommand{\R}{\mathbb{R}}

\newcommand{\rd}{\mathrm{d}}
\definecolor{cadmiumgreen}{rgb}{0.0, 0.42, 0.24}
\numberwithin{equation}{section}
\numberwithin{figure}{section}
\begin{document}

\title[ ]{Convergence of energy minimizers of a MEMS model\\ in the reinforced limit}

\author{Philippe Lauren\c{c}ot}
\address{Institut de Math\'ematiques de Toulouse, UMR~5219, Universit\'e de Toulouse, CNRS \\ F--31062 Toulouse Cedex 9, France}
\email{laurenco@math.univ-toulouse.fr}
\author{Katerina Nik}
\address{Faculty of Mathematics\\ University of Vienna \\ Oskar-Morgenstern-Platz 1 \\ A--1090 Vienna\\ Austria}
\email{katerina.nik@univie.ac.at}
\author{Christoph Walker}
\address{Leibniz Universit\"at Hannover\\ Institut f\" ur Angewandte Mathematik \\ Welfengarten 1 \\ D--30167 Hannover\\ Germany}
\email{walker@ifam.uni-hannover.de}
%
\thanks{Partially supported by the CNRS Projet International de Coop\'eration Scientifique PICS07710}
\date{\today}
\keywords{MEMS, transmission problem, $\Gamma$-convergence, minimizers}
%
%
%
\begin{abstract}
Energy minimizers to a MEMS model with an insulating layer are shown to converge in its reinforced limit to the minimizer of the limiting model as the thickness of the layer tends to zero. The proof relies on the identification of the $\Gamma$-limit of the energy in this limit.
\end{abstract}
%
\maketitle
%
\section{Introduction}

A microelectromechanical system (MEMS), such as an electrostatic actuator,  consists of an elastic plate, which is coated with a thin dielectric layer, clamped on its boundary, and suspended above a rigid ground plate. The latter is also coated with a dielectric layer but with positive thickness $\delta>0$, see Figures~\ref{F1} and~\ref{F2}. Applying a voltage difference between the two plates generates a Coulomb force accross the device and induces a deformation of the elastic plate, thereby changing the geometry of the device and converting electrostatic energy to mechanical energy through a balance between electrostatic and mechanical forces \cite{AmEtal, BG01, FMCCS05, PeB03}. Assuming that the physical state of the MEMS device is fully described by the vertical deflection $u$ of the elastic plate and the electrostatic potential $\psi$ inside the device, a mathematical model is derived in \cite{LW19}. It characterizes equilibrium configurations of the device as critical points of the total energy which is the sum of the mechanical and electrostatic energies, with an additional constraint stemming from the property that the elastic plate cannot penetrate the layer covering the ground plate. Specifically, ignoring variations in the transverse horizontal direction, we consider a two-dimensional MEMS in which the rigid ground plate and the undeflected elastic plate have the same one-dimensional shape $D:=(-L,L)$ with $L>0$. The ground plate is located at height $z=-H-\delta$,  where $H>0$, and is coated with a dielectric layer 
\begin{equation*}
\mathcal{R}_\delta  :=  D\times (-H-\delta,-H)
\end{equation*}
 of positive thickness  $\delta$. The vertical deflection $u$ of the elastic plate from its rest position at $z=0$ is a function from $D$ to $[-H, \infty)$ with $u(\pm L)=0$, so that the elastic plate is described by the graph $\{ (x,u(x))\ :\ x\in D\}$ of the function $u$. Observe that the required lower bound $u\ge -H$ on $u$ is due to the assumption that the elastic plate cannot penetrate the dielectric layer $\mathcal{R}_\delta$, while the boundary conditions $u(\pm L)=0$  reflect the fact that the elastic plate is clamped on its boundary. We then define
$$
 \Omega(u):=\left\{(x,z)\in D\times \mathbb{R}\,:\, -H<  z <  u(x)\right\}
$$
as the free space between the elastic plate and the top of the dielectric layer and denote the interface separating the free space and the dielectric layer by
$$
\Sigma (u):=\{(x,-H)\,:\, x\in D,\, u(x)>-H\}\,.
$$
As for the electrostatic potential $\psi$, it is defined in the full device
\begin{equation*}
	\Omega_\delta({u}):=\left\{(x,z)\in D\times \mathbb{R} \,:\, -H-\delta<  z <  u(x)\right\}=\mathcal{R}_\delta \cup  \Omega( {u})\cup  \Sigma(u)\,.
\end{equation*}
%
\begin{figure}
	\begin{tikzpicture}[scale=0.9]
		\draw[black, line width = 1.5pt, dashed] (-7,0)--(7,0);
		\draw[black, line width = 2pt] (-7,0)--(-7,-5);
		\draw[black, line width = 2pt] (7,-5)--(7,0);
		\draw[black, line width = 2pt] (-7,-5)--(7,-5);
		\draw[black, line width = 2pt] (-7,-4)--(7,-4);
		\draw[black, line width = 2pt, fill=gray, pattern = north east lines, fill opacity = 0.5] (-7,-4)--(-7,-5)--(7,-5)--(7,-4);
		\draw[cadmiumgreen, line width = 2pt] plot[domain=-7:7] (\x,{-1-cos((pi*\x/7) r)});
		\draw[cadmiumgreen, line width = 1pt, arrows=->] (3,0)--(3,-1.15);
		\node at (3.3,-0.6) {${\color{cadmiumgreen} u}$};
		\node[draw,rectangle,white,fill=white, rounded corners=5pt] at (2,-4.5) {$\Omega_1$};
		\node at (2,-4.5) {$\mathcal{R}_\delta $};
		\node at (-2,-3) {${\color{cadmiumgreen} \Omega(u)}$};
		\node at (0,-5.75) {$D$};
		\draw[black, line width = 1pt] (-7,-5.25)--(7,-5.25);
		\draw[black, line width = 1pt] (-7,-5.1)--(-7,-5.4);
		\node at (-7,-5.7) {$-L$};
		\draw[black, line width = 1pt] (7,-5.1)--(7,-5.4);
		\node at (7,-5.7) {$-L$};
		\node at (-7.8,1) {$z$};
		\draw[black, line width = 1pt, arrows = ->] (-7.5,-6)--(-7.5,1);
		\node at (-8.4,-5) {$-H-\delta$};
		\draw[black, line width = 1pt] (-7.6,-5)--(-7.4,-5);
		\node at (-8,-4) {$-H$};
		\draw[black, line width = 1pt] (-7.6,-4)--(-7.4,-4);
		\node at (-7.8,0) {$0$};
		\draw[black, line width = 1pt] (-7.6,0)--(-7.4,0);
	\end{tikzpicture}
	\caption{Geometry of $\Omega_\delta(u)$ for a state $u$ with empty coincidence set $\mathcal{C}(u)$.}\label{F1}
\end{figure}
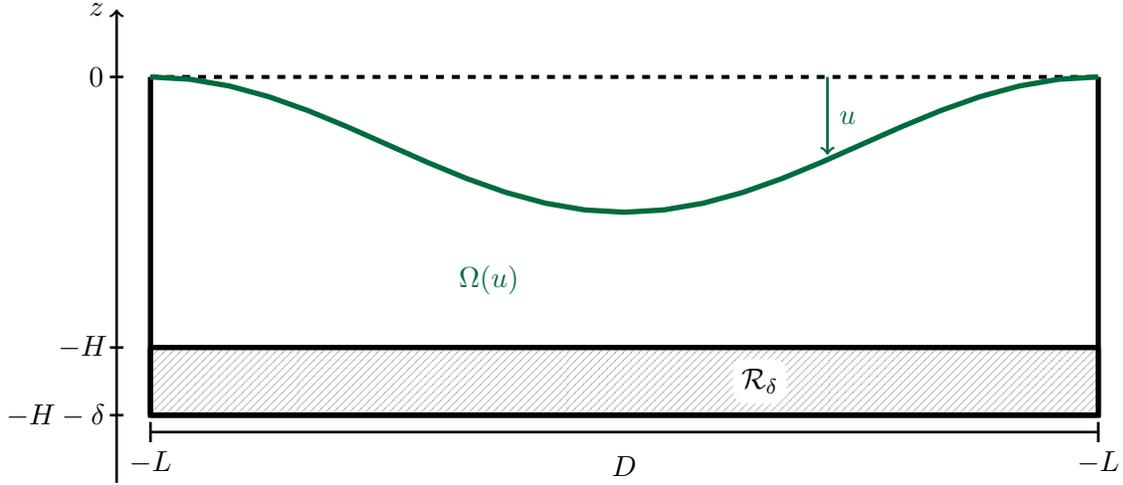

\begin{figure}
	\begin{tikzpicture}[scale=0.9]
	\draw[black, line width = 1.5pt, dashed] (-7,0)--(7,0);
	\draw[black, line width = 2pt] (-7,0)--(-7,-5);
	\draw[black, line width = 2pt] (7,-5)--(7,0);
	\draw[black, line width = 2pt] (-7,-5)--(7,-5);
	\draw[black, line width = 2pt] (-7,-4)--(7,-4);
	\draw[black, line width = 2pt, fill=gray, pattern = north east lines, fill opacity = 0.5] (-7,-4)--(-7,-5)--(7,-5)--(7,-4);
	\draw[blue, line width = 2pt] plot[domain=-7:-3] (\x,{-2-2*cos((pi*(\x+3)/4) r)});
	\draw[blue, line width = 2pt] (-3,-4)--(1,-4);
	\draw[blue, line width = 2pt] plot[domain=1:7] (\x,{-2-2*cos((pi*(\x-1)/6) r)});
	\draw[blue, line width = 1pt, arrows=->] (-5,0)--(-5,-1.8);
	\node at (-4.7,-1) {${\color{blue} u}$};
	\node[draw,rectangle,white,fill=white, rounded corners=5pt] at (2,-4.5) {$\Omega_1$};
	\node at (2,-4.5) {$\mathcal{R}_\delta $};
	\node at (6,-2) {$\color{blue} \Omega(u)$};
	\node at (-6,-2) {$\color{blue} \Omega(u)$};
	\node at (0,-5.75) {$D$};
	\draw[black, line width = 1pt] (-7,-5.25)--(7,-5.25);
	\draw[black, line width = 1pt] (-7,-5.1)--(-7,-5.4);
	\node at (-7,-5.7) {$-L$};
	\draw[black, line width = 1pt] (7,-5.1)--(7,-5.4);
	\node at (7,-5.7) {$-L$};
	\node at (7.9,-3) {${\color{blue} \Sigma(u)}$};
	\draw (7.4,-3) edge[->,bend right, blue, line width = 1pt] (5.2,-3.9);
	\node at (-7.8,1) {$z$};
	\draw[black, line width = 1pt, arrows = ->] (-7.5,-6)--(-7.5,1);
	\node at (-8.4,-5) {$-H-\delta$};
	\draw[black, line width = 1pt] (-7.6,-5)--(-7.4,-5);
	\node at (-8,-4) {$-H$};
	\draw[black, line width = 1pt] (-7.6,-4)--(-7.4,-4);
	\node at (-7.8,0) {$0$};
	\draw[black, line width = 1pt] (-7.6,0)--(-7.4,0);
	\node at (1,-3) {${\color{blue} \mathcal{C}(u)}$};
	\draw (0.45,-3) edge[->,bend right,blue, line width = 1pt] (-0.5,-3.95);
\end{tikzpicture}
	\caption{Geometry of $\Omega_\delta(u)$ for a state $u$ with non-empty coincidence set $\mathcal{C}(u)$.}\label{F2}
\end{figure}
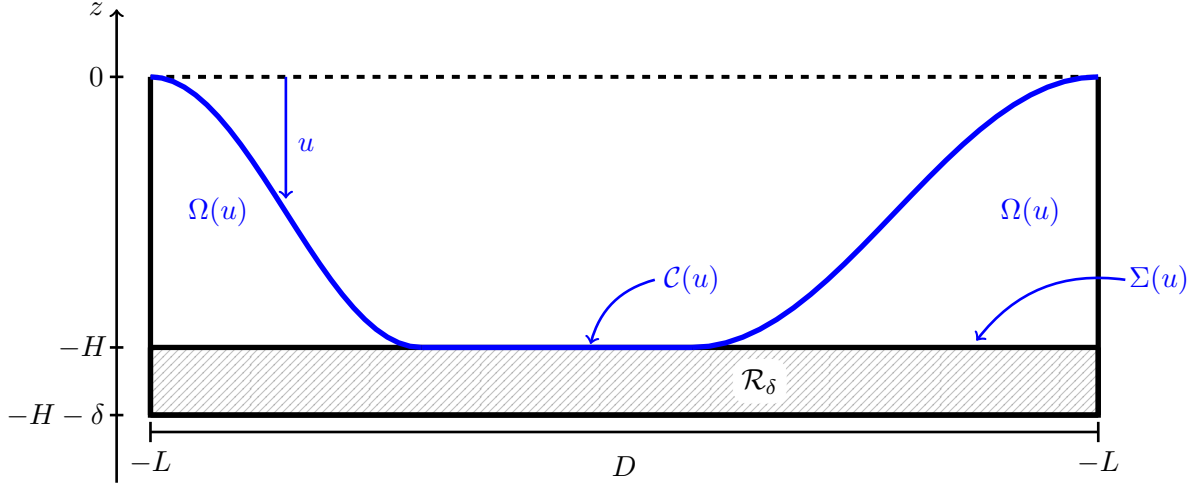
%
 It is worth mentioning at this point that the geometry of the full device $\Omega_\delta(u)$ has different properties according to the minimal value of $u$. Indeed, the free space $\Omega(u)$ is connected and $\Sigma(u)=D\times\{-H\}$ when $\min_{D} u>-H$, while it is disconnected when  $\min_{D} u = -H$, which corresponds to a touchdown of the elastic plate on the dielectric layer $\mathcal{R}_\delta$ on the \textit{coincidence set}
\begin{equation}
	\mathcal{C}(u):=\{x\in D\, :\, u(x)=-H\}\,, \label{cs}
\end{equation}
see Figures~\ref{F1} and~\ref{F2}. In the model derived in \cite{LW19}, equilibrium configurations of the above described MEMS device are critical points of the \textit{total energy} given by
\begin{equation}
E_\delta(u):= E_m(u)+E_{e,\delta}(u)\,. \label{te}
\end{equation}
In \eqref{te}, $E_m(u)$ is the \textit{mechanical energy} 
$$
E_m(u):=\frac{\beta}{2}\|\partial_x^2u\|_{L_2(D)}^2 +\left(\frac{\tau}{2}+\frac{a}{4}\|\partial_x u\|_{L_2(D)}^2\right)\|\partial_x u\|_{L_2(D)}^2
$$
with $\beta>0$, $a\ge 0$,  and $\tau\ge 0$, and includes bending and external stretching effects of the elastic plate. The \textit{electrostatic energy} is
$$
E_{e,\delta}(u):=-\frac{1}{2}\int_{\Omega_\delta(u)}  \sigma_\delta \vert\nabla \psi_{u,\delta}\vert^2\,\rd (x,z)\,,
$$
with $\sigma_\delta$ denoting the permittivity of the device (see \eqref{sigma} below), and $ \psi=\psi_{u,\delta}$ is the electrostatic potential satisfying the transmission problem
\begin{subequations}\label{TMP}
\begin{align}\label{e1}
\mathrm{div}(\sigma_\delta\nabla\psi_{u,\delta})&=0 \quad\text{in }\ \Omega_\delta(u)\,, \\
\llbracket \psi_{u,\delta} \rrbracket = \llbracket \sigma_\delta \partial_z \psi_{u,\delta}\rrbracket &=0\quad\text{on }\  \Sigma(u)\,,\label{e2} \\
\psi_{u,\delta} &=h_{u,\delta}\quad\text{on }\ \partial\Omega_\delta(u)\,.\label{e3}
\end{align}
\end{subequations}
Here, $\llbracket\cdot \rrbracket$  denotes the jump of a function across the interface $\Sigma(u)$. The boundary values of the electrostatic potential are prescribed by a function  $h_{u,\delta}$ which satisfies the assumptions listed below in \eqref{bobbybrown}. A specific example, when $\sigma$ does not depend on the vertical coordinate $z$, is 
	\begin{equation}
		h_{u,\delta}(x,z) = \left\{
		\begin{split}
			 & \frac{1+\sigma(x)(H+z)}{1+\sigma(x)(H+u(x))}\,, \qquad (x,z)\in \bar{D}\times [-H,\infty)\,, \\  
			& \frac{1}{\delta}  \frac{z+H+\delta}{1+\sigma(x)(H+u(x))}\,, \qquad (x,z)\in \bar{D}\times [-H-\delta,-H]\,.
		\end{split}
	\right. \label{zz1}
	\end{equation}

Since the elastic plate is clamped at the boundary and cannot penetrate the dielectric layer $\mathcal{R}_\delta$, the set of admissible deflections is
\begin{equation*}
\bar{S}_0 := \left\{ u\in H_D^2(D)\,:\,  u\ge -H \text{ in } D \right\}\,,
\end{equation*}
where
\begin{equation*}
H_D^2(D) := \left\{ u\in H^2(D)\,:\, u(\pm L) = \partial_x u(\pm L)=0 \right\}.
\end{equation*}
Equilibrium configurations of the MEMS device are then critical points $u\in\bar{S}_0$ of the total energy $E_\delta$. Their analysis involves the associated transmission problem \eqref{TMP} solved by the electrostatic potential $\psi_{u,\delta}$. A natural question is what happens when the thickness $\delta$ of the dielectric layer tends to zero,  in particular, whether the reduced model derived in this limit retains the dielectric inhomogeneity of the device. When the dielectric permittivity $\sigma_\delta$ of the device does not depend on $\delta$, the influence of the dielectric layer is lost in the limit $\delta\to 0$, and the reduced model is obtained simply by setting $\delta=0$ in \eqref{te} and \eqref{TMP}, discarding the jump condition \eqref{e2} which is then meaningless. Building upon the outcome of  \cite{AB86, BCF80}, it turns out that it is rather the \textit{reinforced limit}, where the dielectric permittivity scales as $\delta$ in the layer $\mathcal{R}_\delta$, which leads to a relevant reduced model. For a given deflection $u\in \bar{S}_0$, the reinforced limit of the transmission problem \eqref{TMP} is identified in \cite{AMOP20} by a $\Gamma$-convergence approach. More precisely, it is shown in \cite{AMOP20} that the reinforced limit as $\delta\rightarrow 0$ of \eqref{TMP} is
\begin{subequations}\label{MBP0}
	\begin{align}
		\mathrm{div} (\sigma \nabla\psi_{u}) &=0 \quad\text{in }\ \Omega(u)\,, \label{MBP1}\\
		\psi_{u} &=h_u\quad\text{on }\ \partial\Omega(u)\setminus  \Sigma(u)\,,\label{MBP2}\\
		- \partial_z\psi_{u} +\sigma (\psi_{u}-\mathfrak{h}_u)&=0\quad\text{on }\    \Sigma(u)\label{MBP3}\,;
	\end{align}
\end{subequations}
that is, in the reinforced limit the electrostatic potential $\psi_{u}$ solves Laplace's equation in $\Omega(u)$ with a Robin boundary condition along the interface $\Sigma(u)$ and a Dirichlet  condition on the other boundary parts. Here, $\sigma := \sigma_\delta \mathbf{1}_{\Omega(u)}$ is assumed to be independent of $\delta$. The total energy is then given by
\begin{equation}
E(u):= E_m(u)+E_{e,0}(u)\,, \label{rte}
\end{equation}
where
$$
E_{e,0}(u):=-\frac{1}{2}\int_{\Omega(u)} \sigma \vert\nabla \psi_{u}\vert^2\,\rd (x,z)-\frac{1}{2}\int_D\sigma(x,-H)\big\vert\psi_{u}(x,-H)-\mathfrak{h}_u(x)\big\vert^2\,\rd x
$$
and $\mathfrak{h}_u$ is defined below in \eqref{bobbybrown}.

The purpose of this research is to complete the outcome of \cite{AMOP20} by identifying the reinforced limit of the full model and showing that, in this limit, if $u_\delta^*\in \bar{S}_0$ is a minimizer of $E_\delta$ in $\bar{S}_0$ for each $\delta\in (0,1)$, then the cluster points of $(u_\delta^*)_{\delta\in (0,1)}$ in $L_2(D)$ are minimizers of the reduced total energy $E$ in $\bar{S}_0$. The main tool we shall employ in the forthcoming analysis is the theory of $\Gamma$-convergence. We shall actually show that, under suitable assumptions on the dielectric permittivity $\sigma_\delta$ and the boundary  values in \eqref{TMP}, the $\Gamma$-limit in $L_2(D)$ of $(E_\delta)_{\delta\in (0,1)}$ is the reduced total energy $E$ defined in \eqref{rte}.

Let us finally remark that, in this paper, we focus on the energy approach to take into account the influence of the thickness of a dielectric layer as first developed in \cite{LW18} for a related model. We refer to \cite{BG01,LLG14,LLG15, Pel01a} for alternative approaches to model dielectric layers, all designed within the so-called small aspect ratio approximation. Recall that, in the latter, the electrostatic potential is given explicitly as a function of the deflection $u$ and the model then reduces to a single equation for $u$. Such models have been extensively studied in the last decades in the mathematical literature since the pioneering works of \cite{BGP00, FMPS06, GPW05, Pel01a}, see the book \cite{EGG10}, the survey \cite{LW17b} and the references therein.

\section{Convergence of minimizers}\label{sec.mr}

As already mentioned, the reinforcement limit requires that the permittivity $\sigma_\delta$ in the dielectric layer $\mathcal{R}_\delta$ scales with the layer's thickness; that is, the (scaled) permittivity of the device is given in the form
\begin{subequations}\label{sigma}
	\begin{equation}\label{sigmad}
		\sigma_\delta(x,z):=  \left\{ \begin{array}{ll}
			\delta\sigma (x,z)\,, & (x,z)\in \mathcal{R}_\delta \,, \\
			1\,, & (x,z)\in D\times (-H,\infty)\,,
		\end{array} \right.
	\end{equation}
	for $\delta\in (0,1)$, where $\sigma \in C^2(\bar D\times [-H-1,-H])$ is a fixed function with 
	\begin{equation}\label{sigmam}
		\sigma (x,z) >0\,,\quad (x,z)\in \bar D\times [-H-1,-H]\,.
	\end{equation}
\end{subequations}

With this specific form of $\sigma$, we can show that cluster points as $\delta\to 0$ of minimizers of the total energy $E_\delta$ on $\bar{S}_0$ are minimizers of the reduced total energy $E$. More precisely:

\begin{theorem}\label{TT1}
	Suppose that the dielectric permittivity satisfies \eqref{sigma} and that the assumptions on the boundary values in \eqref{e3} are given by \eqref{bobbybrown} below. For $\delta\in (0,1)$ let $u_\delta^*\in \bar{S}_0$ be any minimizer of $E_\delta$ on $\bar{S}_0$ with corresponding electrostatic potential $\psi_{u_\delta^*,\delta}$ satisfying \eqref{TMP}. Then 
	$$
	\sup_{\delta\in (0,1)}\|u_\delta^*\|_{H^2(D)}<\infty \qquad \text{ and }\qquad \sup_{\delta\in (0,1)}\|\psi_{u_\delta^*,\delta}\|_{H^1(\Omega(u_\delta^*))}<\infty\,,
	$$ 
and there are a subsequence $\delta_j\rightarrow 0$ and a minimizer $u^*\in \bar{S}_0$ of $E$ on $\bar{S}_0$ such that 
	$$
	\lim_{j\to\infty} \big\| u_{\delta_j}^* - u^* \big\|_{H^2(D)} = 0
	$$ 
	and 
	$$
	\lim_{j\to\infty} E_{\delta_j}(u_{\delta_j}^*) = E(u^*)\,.
	$$
	Moreover, for $M>0$ such that $-H\le u_\delta^*\le M-H$ a.e., we have
	$$
	\psi_{u_{\delta_j}^*,\delta_j}-h_{u_{\delta_j}^*}\rightharpoonup \psi_{u^*}-h_{u^*}\quad\text{ in }\ H^1(D\times (-H,M))\,,
	$$
	where $\psi_{u^*}$ satisfies~\eqref{MBP0} (with $u$ replaced by $u^*$).
\end{theorem}

As we shall see below, the main step in the proof of Theorem~\ref{TT1} is the $\Gamma$-convergence of the sequence $(E_\delta)_{\delta\in (0,1)}$ in $L_2(D)$ to $E$ which is established in Section~\ref{S3}.  We then combine this property with estimates on the minimizers of $E_\delta$ on $\bar{S}_0$, which do not depend on $\delta\in (0,1)$ and are derived in Sections~\ref{S3.3}-\ref{S3.4} to complete the proof. 

Let us finally point out that the assumptions \eqref{sigma} and \eqref{bobbybrown} on the permittivity $\sigma_\delta $ and the boundary conditions $h_{u,\delta}$ guarantee that, for each $\delta\in (0,1)$, the total energy $E_\delta$ defined in \eqref{te} has at least one minimizer $u_\delta^*\in \bar{S}_0$; that is, 
\begin{equation}\label{B}
	E_\delta(u_\delta^*)=\min_{\bar{S}_0}E_\delta\,,
\end{equation}
see \cite[Theorem~1.3]{JEPE20}. Actually, the corresponding electrostatic potential  $\psi_{u_\delta^*,\delta} \in H^1(\Omega_\delta(u_\delta))$ is a strong solution to the  transmission problem~\eqref{TMP} in the sense that $\psi_{u_\delta^*,\delta}\vert_{\mathcal{R}_\delta} \in H^2(\mathcal{R}_\delta)$ and $\psi_{u_\delta^*,\delta}\vert_{\Omega(u_\delta^*)} \in H^2(\Omega(u_\delta^*))$, this regularity property being in fact true for any $u\in \bar{S}_0$ \cite{LW19, JEPE20}. 

As for the reduced total energy $E$, the existence of minimizers of $E$ on $\bar{S}_0$ has already been established in \cite[Theorem~2.3]{LNW20} by a direct approach, assuming additionally that
\begin{equation}
		\partial_z h(x,-H,w) = \sigma(x) \big[ h(x,-H,w) - \mathfrak{h}(x,w) \big]\,, \qquad (x,w)\in D\times [-H,\infty)\,, \label{bryanadams}
\end{equation}
besides \eqref{bobbybrown} below. Theorem~\ref{TT1} then extends the existence of minimizers of $E$ on $\bar{S}_0$ to the situation where \eqref{bryanadams} does not hold. However, it does not provide the $H^2$-regularity of the associated electrostatic potential $\psi_{u^*}$ solving \eqref{MBP0}, which is shown to be true in \cite[Theorem~2.2]{LNW20} under the assumptions \eqref{bryanadams} and \eqref{bobbybrown}.

\section{Assumptions and auxiliary results}\label{S2}

This section is devoted to a precise definition of the boundary conditions \eqref{e3} and~\eqref{MBP3}, and includes as well useful properties of $h_{u,\delta}$ on which we rely on in the sequel.

\subsection{ Boundary data}

We fix two $C^2$-functions
\begin{subequations}\label{bobbybrown}
\begin{equation}\label{bobbybrown2a}
h_b: \bar{D}\times [-H-1,-H]\times [-H,\infty)\rightarrow \R
\end{equation}
and 
\begin{equation}\label{bobbybrown2aa}
h: \bar{D}\times [-H,\infty)\times [-H,\infty)\rightarrow \R
\end{equation}
satisfying
\begin{align}
h_b(x,-H,w)&=h(x,-H,w)\,,\quad (x,w)\in \bar D\times [-H,\infty)\,,\label{bobbybrown2}\\
  \sigma(x,-H)\partial_z h_b(x,-H,w)& = \partial_z h(x,-H,w)\,,\quad (x,w)\in \bar D\times [-H,\infty)\,.\label{bobbybrown3}
\end{align}
We then define for $(x,w) \in \bar D \times [-H,\infty)$
\begin{equation}\label{bobbybrown40}
h_\delta(x,z,w):=  \left\{ \begin{array}{ll}
\displaystyle{h_b\left(x,-H+\frac{z+H}{\delta},w\right)}\,, & z \in [-H-\delta,-H)\,, \\ \vspace{-3mm}
\\
h(x,z,w)\,, &  z\in  [-H,\infty)\,,
\end{array} \right.
\end{equation}
 and observe that $h_\delta\in C(\bar D \times [-H-\delta,\infty)\times [-H,\infty))$ by \eqref{bobbybrown2}.

 In order to guarantee the coercivity of the energy functional $E_{\delta}$ we require that there is a constant $m>0$ such that
\begin{equation}
\vert \partial_x h_b(x,z,w)\vert +\vert\partial_z h_b(x,z,w)\vert  \le \sqrt{m(1+ w^2)}\,, \quad \vert\partial_w h_b(x,z,w)\vert \le \sqrt{m}\,, \label{bobbybrown5}
\end{equation}
for $(x,z,w)\in \bar D \times [-H-1,-H] \times [-H,\infty)$ and
\begin{equation} 
\vert \partial_x h(x,z,w)\vert +\vert\partial_z h(x,z,w)\vert \le \sqrt{\frac{m (1+ w^2)}{H+w}}\,,\quad \vert\partial_w h(x,z,w)\vert \le \sqrt{\frac{m}{H+w}}\,,\label{bobbybrown6}
\end{equation}
for $(x,z,w)\in \bar D \times [-H,\infty) \times [-H,\infty)$. Moreover, we assume that
\begin{equation}\label{Kbound0}
\partial_w h_b(x,-H-1,w) =0\,,\quad (x,w)\in \bar{D}\times [-H,\infty)\,,
\end{equation}
and that there is $K>0$ such that
\begin{equation}\label{Kbound}
\vert \partial_x h(x,w,w)\vert + \vert \partial_z h(x,w,w)+\partial_w h(x,w,w)\vert \le K\,,\quad (x,w)\in \bar D\times [-H,\infty)\,.
\end{equation}
Given a function $u:\bar D\rightarrow [-H,\infty)$ we shall also use the abbreviations
\begin{equation}\label{h000}
h_{u,\delta}(x,z):= h_\delta(x,z,u(x))\,,\quad (x,z)\in \Omega_\delta(u)\,,
\end{equation}
and 
\begin{equation}\label{h00}
h_{u}(x,z):=h(x,z,u(x))\,,\quad (x,z)\in \Omega(u)\,.
\end{equation}
Furthermore, we set
\begin{equation}\label{h0}
\mathfrak{h}_u(x):=\mathfrak{h}_u(x,-H):=h_b(x,-H-1,u(x))\,,\quad x\in \bar D\,.
\end{equation}
\end{subequations}
Note that \eqref{bobbybrown2}-\eqref{bobbybrown3} imply that  $h_{u,\delta}$ satisfies the transmission conditions \eqref{e2}:
\begin{equation*}\label{bobbybrown44}
\llbracket h_{u,\delta} \rrbracket = \llbracket \sigma_\delta \partial_z h_{u,\delta}\rrbracket =0\quad\text{on }\  \Sigma(u)\,.
\end{equation*}

Simple computations show that the example provided in \eqref{zz1} satisfies \eqref{bobbybrown} with 
\begin{equation*}
	h_b(x,z,w) = \frac{z+H+1}{1+\sigma(x)(H+w)}\,, \qquad (x,z,w)\in \bar{D}\times [-H-1,-H]\times [-H,\infty)\,,
\end{equation*}
and
\begin{equation*}
	h(x,z,w) = \frac{1+\sigma(x)(H+z)}{1+\sigma(x)(H+w)}\,, \qquad (x,z,w)\in \bar{D}\times [-H,\infty) \times [-H,\infty)\,.
\end{equation*}

\subsection{Auxiliary results}

We begin with some  properties of the function $h_{u,\delta}$ that we derive from assumptions \eqref{sigma}-\eqref{bobbybrown} imposed above. For further use, we set 
$$
\sigma_{max}:=1 + \max_{\bar D\times [-H-1,-H]}\sigma\,.
$$

\begin{lemma}\label{LL0}
Assume \eqref{sigma} and \eqref{bobbybrown}.

\noindent \textbf{(i)} There is a constant $c_0>0$ depending on $m$, $L$, and $\sigma_{max}$ such that, given $u\in \bar{S}_0$ and $\delta\in (0,1)$, 
\begin{equation}\label{H1}
\int_{\Omega_\delta(u)} \sigma_\delta \vert \nabla h_{u,\delta}\vert^2\, \rd (x,z) \le c_0\big(1+\|u\|_{L_2(D)}^2+ \| \partial_x u\|_{L_2(D)}^2\big)\,.
\end{equation}

\noindent \textbf{(ii)} Suppose that $u_\delta\rightarrow u$ in $H^1(D)$ as $\delta\rightarrow 0$ and that $-H\le u_\delta$ in $D$. Then $$M:= \sup_{\delta\in (0,1)} \|u_\delta\|_{L_\infty(D)}<\infty$$ and, as $\delta\rightarrow 0$,
\begin{subequations}
\begin{align}
&h_{u_\delta,\delta} \to h_{u}\quad\text{in}\quad   H^1(D\times (-H,M))\,,\label{s1}\\
&\mathfrak{h}_{u_\delta}  \to \mathfrak{h}_{u} \quad\text{in}\quad L_2(D)\,,\label{s2}\\
&h_{u_\delta}(.,-H) \to h_{u}(.,-H) \quad\text{in}\quad L_2(D)\,.\label{s3}
\end{align}
Moreover,
\begin{equation}\label{24d}
\lim_{\delta\to 0} \int_{\Omega(u_\delta)} |\nabla h_{u_\delta,\delta}|^2 d(x,z) = \int_{\Omega(u)} |\nabla h_{u}|^2 d(x,z)\,.
\end{equation}
\end{subequations}
\end{lemma}

\begin{proof}
 \textit{\textbf{(i)}} Using \eqref{sigma}, \eqref{bobbybrown40}, \eqref{h000}, and the definition of $\Omega_\delta(u)$ we have
\begin{equation*}
\begin{split}
\int_{\Omega_\delta(u)} \sigma_\delta \vert \nabla h_{u,\delta}\vert^2\, \rd (x,z) 
&=\int_{\Omega(u)} \vert \partial_x h(x,z,u(x))+\partial_x u(x) \partial_w h(x,z,u(x))\vert^2\, \rd (x,z)\\
& \quad +\int_{\Omega(u)} \vert \partial_z h(x,z,u(x))\vert^2\, \rd (x,z)\\
&\quad 
+\delta \int_{\mathcal{R}_\delta} \sigma(x,z)  \left\vert  \partial_x h_b\left(x,-H+\frac{z+H}{\delta}, u(x)\right)\right.\\
&\qquad\qquad\qquad\qquad \left.+\partial_x u(x) \partial_w h_b\left(x,-H+\frac{z+H}{\delta}, u(x)\right)\right\vert^2\, \rd (x,z)\\
&\quad 
+ \delta \int_{\mathcal{R}_\delta} \sigma(x,z)  \left\vert \frac{1}{\delta} \partial_z h_b\left(x,-H+\frac{z+H}{\delta}, u(x)\right)\right\vert^2\, \rd (x,z)
\,.
\end{split}
\end{equation*}
Invoking \eqref{bobbybrown5}-\eqref{bobbybrown6}  and  Young's inequality, we derive
\begin{equation*}
\begin{split}
\int_{\Omega_\delta(u)} \sigma_\delta \vert \nabla h_{u,\delta}\vert^2\, \rd (x,z) & \le 2m\int_{\Omega(u)} \frac{1+ u(x)^2+(\partial_x u(x))^2}{H+u(x)}\,\rd (x,z)\\
&\qquad + m \int_{\Omega(u)} \frac{1+ u(x)^2}{H+u(x)}\,\rd (x,z)\\
&\qquad +2m \delta\sigma_{max}\int_{\mathcal{R}_\delta}\left(1+ u(x)^2 + (\partial_x u(x))^2\right)\,\rd (x,z)\\
&\qquad  +\frac{m \sigma_{max}}{\delta}\int_{\mathcal{R}_\delta}\left( 1+ u(x)^2\right)\,\rd (x,z)\\
&\le m (1+\sigma_{max})\left[3\left(\vert D\vert +\|u\|_{L_2(D)}^2\right)+ 2 \| \partial_x u\|_{L_2(D)}^2\right]\,.
\end{split}
\end{equation*}
This proves \textit{\textbf{(i)}}.

\smallskip 

\noindent \textit{\textbf{(ii)}} First, $M$ is well-defined and finite owing to the continuous embedding of $H^1(D)$ in $C(\bar D)$ and the strong convergence of $(u_\delta)_{\delta\in (0,1)}$ in $H^1(D)$. Next, the stated convergences readily follow from the smoothness of $h$ and $h_b$, from the convergence of $u_\delta\rightarrow u$ in $H^1(D)$, and the continuous embedding of $H^1(D)$ in $C(\bar{D})$.
\end{proof}

\section{Convergence of minimizers}\label{S3}

Three steps are needed to prove Theorem~\ref{TT1}: we begin by establishing in Section~\ref{S3.1} the convergence of the electrostatic energy $E_{e,\delta}$ as $\delta\to 0$, building upon the analysis performed in \cite{AMOP20} for a reduced problem. This convergence, along with the weak lower semicontinuity of the mechanical energy $E_m$, leads us to the $\Gamma$-convergence of $E_\delta$ to $E$ in $L_2(D)$, see Section~\ref{S3.2}. Such a property provides information on the relationship between minimizers for the cases $\delta>0$ and $\delta=0$, which we use in Sections~\ref{S3.3}--\ref{S3.4} to complete the proof of Theorem~\ref{TT1}.

\subsection{Convergence of the electrostatic energy}\label{S3.1}

 Building upon the analysis performed in \cite{AMOP20}, we investigate the limit of the electrostatic energy $E_{e,\delta}$ as $\delta\to 0$. Recalling that the main outcome of \cite{AMOP20} is that 
\begin{equation*}
\lim_{\delta\to 0} E_{e,\delta}(u) = E_{e,0}(u)
\end{equation*}
for any $u\in \bar{S}_0$, we extend this result to a sequence $(u_\delta)_{\delta\in (0,1)}$ in $\bar{S}_0$ and show that $(E_{e,\delta}(u_\delta))_{\delta\in (0,1)}$ converges to $E_{e,0}(u)$ as $\delta\to 0$, provided that $(u_\delta)_{\delta\in (0,1)}$ converges to $u$ in $H^1(D)$. More precisely, consider a sequence $(u_\delta)_{\delta\in (0,1)}$ in $\bar{S}_0$ and $u\in \bar{S}_0$ such that
\begin{subequations}\label{i}
\begin{equation}\label{ia}
u_\delta\to u\ \text{ in }\ H^1(D)\  \text{ as }\  \delta\to 0\,,\qquad -H\le u_\delta(x)\,, \quad x\in D\,.
\end{equation}
Observe that the convergence \eqref{ia} and the continuous embedding of $H^1(D)$ in $C(\bar{D})$ ensure that
\begin{equation}
0 \le H + u_\delta(x) \,, H + u(x) \le M:= \sup_{\delta\in (0,1)} \|H+u_\delta\|_{L_\infty(D)}\,,\quad x\in D\,. \label{ib}
\end{equation}
\end{subequations}

\begin{proposition}\label{PP1}
Suppose \eqref{i}  and set $\Omega(M):=D\times (-H,M)$. Then 
$$
\lim_{\delta\to 0} E_{e,\delta}(u_\delta)=E_{e,0} (u)
$$
and
$$
\psi_{u_\delta,\delta}-h_{u_\delta,\delta}\longrightarrow \psi_{u}-h_{u} \quad\text{in }\ L_2(\Omega(M))\quad \text{as }\ \delta\to 0\,.
$$
\end{proposition}

\begin{proof}
We use a $\Gamma$-convergence approach combining arguments from \cite[Proposition~4.1]{LNW20} and \cite[Theorem~3.1]{AMOP20}. Let $\mathcal{O}_M := D\times (-H-1,M)$ and define, for $\delta\in (0,1)$,
$$
G_\delta[\vartheta]:=\left\{\begin{array}{ll} \dfrac{1}{2}\displaystyle\int_{ \Omega_\delta(u_\delta)} \sigma_\delta \vert\nabla (\vartheta+h_{u_\delta,\delta})\vert^2\,\rd (x,z)\,, & \vartheta\in H_0^1(\Omega_\delta(u_\delta))\,,\\
\infty\,, & \vartheta\in L_2(\mathcal{O}_M)\setminus H_0^1(\Omega_\delta(u_\delta))\,.
\end{array}\right. 
$$
Then
\begin{equation}\label{G1}
E_{e,\delta}(u_\delta)=-G_\delta[\chi_{u_\delta,\delta}] \quad \text{ with }\quad \chi_{u_\delta,\delta}:=\psi_{u_\delta,\delta}-h_{u_\delta,\delta}\in H_0^1(\Omega_\delta(u_\delta))\,,
\end{equation}
and $\chi_{u_\delta,\delta}$ is the unique minimizer of $G_\delta$ on $H_0^1(\Omega_\delta(u_\delta))$, see \cite[Proposition~3.1]{LW19}.

We next introduce $H_B^1(\Omega(u))$ as the closure in $H^1(\Omega(u))$ of the set
\begin{equation*}
\begin{split}
C_B^1(\overline{\Omega(u)}):=\Big\{\vartheta\in C^1(\overline{\Omega(u)}) : \ 
&\vartheta(x,u(x))=0\,,\ x\in D\\[-0.1cm]
& \text{ and }\vartheta(x,z)=0\,,\ (x,z)\in  \{\pm L\}\times (-H,0] \Big\}\,.
\end{split}
\end{equation*}
Noticing that $\vartheta(x,u(x))=\vartheta(x,-H)=0$ for $x\in \mathcal{C}(u)$ and $\vartheta\in C_B^1(\overline{\Omega(u)})$, we agree upon setting $\vartheta(x,u(x))=\vartheta(x,-H):=0$ for all $x\in \mathcal{C}(u)$ and $\vartheta\in H_B^1(\Omega(u))$ in the sequel. 
Now, given $\vartheta\in H_B^1(\Omega(u))$, we define
\begin{equation}\label{G}
G[\vartheta]:= \frac{1}{2}\int_{ \Omega(u)}   \big\vert\nabla (\vartheta+h_u)\big\vert^2\,\rd (x,z) + \frac{1}{2}\int_{ D} \big(\sigma \big\vert \vartheta+h_u-\mathfrak{h}_u\big\vert^2\big)(x,-H)\,\rd x \,,
\end{equation}
with $\mathfrak{h}_u$ defined in \eqref{h0},  and
$$
G[\vartheta]:= \infty\,,\quad \vartheta\in L_2(\mathcal{O}_M)\setminus H_B^1(\Omega(u))\,.
$$
Then
\begin{equation}\label{G2}
E_{e,0}(u)=-G[\chi_{u}] \quad\text{with}\quad \chi_{u}:=\psi_{u}-h_{u}\in H_B^1(\Omega(u))\,,
\end{equation}
and $\chi_{u}$ is the unique minimizer of $G$ on $H_B^1(\Omega(u))$, see \cite[Proposition~3.3]{AMOP20}. 
We now claim that
\begin{equation}\label{GG}
\Gamma-\lim_{\delta\rightarrow 0} G_\delta =G\quad\text{in }\   L_2(\mathcal{O}_M)\,.
\end{equation}
For \eqref{GG} we have to prove the \textit{asymptotic weak lower semicontinuity} and the existence of a \textit{recovery sequence}.

\bigskip

\noindent \textit{(i) Asymptotic weak lower semicontinuity.} Consider $\vartheta_0\in L_2(\mathcal{O}_M)$ and a sequence $(\vartheta_\delta)_{\delta\in (0,1)}$ in $L_2(\mathcal{O}_M)$ such that
\begin{equation}\label{k1}
\vartheta_\delta\rightarrow \vartheta_0\quad \text{in}\quad L_2(\mathcal{O}_M)\,.
\end{equation} 
 We shall then show that
\begin{equation}
G[\vartheta_0]\le\liminf_{\delta\rightarrow 0} G_\delta[\vartheta_\delta]\,. \label{AWLS}
\end{equation}
Due to the definitions of $G_\delta$ and $G$, we only need to consider the case where $\vartheta_\delta\in H_0^1(\Omega_\delta(u_\delta))$ for $\delta\in (0,1)$ and
\begin{equation}\label{1}
\sup_{\delta\in (0,1)} G_\delta[\vartheta_\delta]<\infty\,.
\end{equation}
We may then extend $\vartheta_\delta$ trivially to $\Omega(M)=D\times (-H,M)$, so that $\vartheta_\delta\in H_B^1(\Omega(M))$. We next infer from \eqref{sigma} and the definition of $G_\delta$ that
\begin{align*}
\int_{\Omega(M)} |\nabla\vartheta_\delta|^2\,\rd (x,z) 
& = \int_{\Omega(u_\delta)} \sigma_\delta |\nabla\vartheta_\delta|^2\,\rd (x,z) 
\\
& \le 2 \int_{\Omega(u_\delta)} \sigma_\delta |\nabla(\vartheta_\delta+h_{u_\delta,\delta})|^2 \, \rd(x,z) + 2 \int_{\Omega(u_\delta)} \sigma_\delta |\nabla h_{u_\delta,\delta}|^2 \, \rd(x,z)
\\
& \le 2  G_\delta[\vartheta_\delta]+ 2 \int_{\Omega(u_\delta)} \sigma_\delta |\nabla h_{u_\delta,\delta}|^2\,\rd (x,z)\,,
\end{align*}
and the right-hand side of the above inequality is bounded by \eqref{i}, \eqref{1}, and Lemma ~\ref{LL0}. Consequently, taking also into account \eqref{k1} and the property $\vartheta_\delta\in H_0^1(\Omega_\delta(u_\delta))$ for $\delta\in (0,1)$, we conclude that $(\vartheta_\delta)_{\delta\in (0,1)}$ is bounded in $H_B^1(\Omega(M))$.  Owing to \eqref{i} and Lemma~\ref{LL0}, we may assume without loss of generality that
\begin{equation}
\vartheta_\delta +h_{u_\delta,\delta} \rightharpoonup \vartheta_0 + h_u \quad \text{ in }\ H^1(\Omega(M))\,. \label{xp}
\end{equation}
Moreover, since $\Omega(M)$ is a Lipschitz domain, the embedding of $H^1(\Omega(M))$ in $H^{3/4}(\Omega(M))$ is compact, see \cite[Theorem~1.4.3.2]{Gr85}, while the trace operator is continuous from $H^{3/4}(\Omega(M))$ in $L_2(\partial\Omega(M))$, see \cite[Theorem~1.5.1.2]{Gr85}. We may thus assume without loss of generality that
\begin{equation}\label{x}
\vartheta_\delta\rightarrow \vartheta_0\quad\text{ in }\ L_2(\partial\Omega(M))\,.
\end{equation}
In particular, 
\begin{equation}\label{xiiia}
\vartheta_\delta(\cdot,-H)\rightarrow \vartheta_0 (\cdot,-H)\quad\text{in }\ L_2\big(D\big)\,,
\end{equation}
and it follows from \eqref{i}, \eqref{xiiia}, and Lemma~\ref{LL0} that
\begin{equation}\label{xi}
\begin{split}
\lim_{\delta\to 0} &\int_{D}\big(\sigma \big\vert\vartheta_\delta +h_{u_\delta} -\mathfrak{h}_{u_\delta}\big\vert^2\big)(x,-H)\,\rd x=
\int_{D}\big(\sigma\big\vert\vartheta_0+h_u -\mathfrak{h}_u\big\vert^2\big)(x,-H)\,\rd x\,.
\end{split}
\end{equation}
Next, arguing as in \cite[Proposition~4.1]{LNW20}, we deduce from \eqref{xp} and Lemma~\ref{LL0} that
\begin{align*}
& \liminf_{\delta\rightarrow 0} \int_{\Omega(u_\delta)}\vert\nabla(\vartheta_\delta+h_{u_\delta,\delta})\vert^2\,\rd (x,z) \\
& \qquad = \liminf_{\delta\rightarrow 0} \int_{\Omega(M)}\vert\nabla(\vartheta_\delta+h_{u_\delta,\delta})\vert^2\,\rd (x,z) - \lim_{\delta\rightarrow 0} \int_{\Omega(M)\setminus\Omega(u_\delta)}\vert\nabla h_{u_\delta,\delta}\vert^2\,\rd (x,z) \\
& \qquad \ge \int_{\Omega(M)}\vert\nabla(\vartheta_0+h_{u})\vert^2\,\rd (x,z) - \int_{ \Omega(M)\setminus \Omega(u)}\vert\nabla h_{u}\vert^2\,\rd (x,z) \\
& \qquad = \int_{\Omega(u)}\vert\nabla(\vartheta_0+h_{u})\vert^2\,\rd (x,z) \,.
\end{align*}
Hence, together with \eqref{xi},
\begin{equation}\label{xii}
\begin{split}
\liminf_{\delta\rightarrow 0} &\left\{\frac{1}{2}\int_{\Omega(u_\delta)}\vert\nabla(\vartheta_\delta+h_{u_\delta,\delta})\vert^2\,\rd (x,z) +\frac{1}{2}\int_{D}\big(\sigma \big\vert\vartheta_\delta +h_{u_\delta} -\mathfrak{h}_{u_\delta}\big\vert^2\big)(x,-H)\,\rd x\right\} \\
&\ge \frac{1}{2}\int_{\Omega(u)}\vert\nabla(\vartheta_0+h_{u})\vert^2\,\rd (x,z) +\frac{1}{2}\int_{D}\big(\sigma \big\vert\vartheta_0 +h_{u} -\mathfrak{h}_{u}\big\vert^2\big)(x,-H)\,\rd x \,.
\end{split}
\end{equation}
Moreover, \eqref{1} entails that
\begin{equation*}
\sup_{\delta\in (0,1)}\int_{\mathcal{R}_\delta}\sigma_\delta \big\vert\nabla(\vartheta_\delta +h_{u_\delta,\delta})\vert^2\,\rd (x,z)<\infty\,.
\end{equation*}
The continuity of $\sigma$ now warrants that
\begin{equation*}
\begin{split}
\liminf_{\delta\rightarrow 0} \, &\int_{\mathcal{R}_\delta }\sigma_\delta(x,z)\vert\nabla(\vartheta_\delta+h_{u_\delta,\delta})\vert^2\,\rd (x,z)\\
&= \liminf_{\delta\rightarrow 0}\, \delta\int_{\mathcal{R}_\delta }\sigma(x,-H)\vert\nabla(\vartheta_\delta+h_{u_\delta,\delta})\vert^2\,\rd (x,z)\\
&\ge \liminf_{\delta\rightarrow 0} \, \delta\int_{\mathcal{R}_\delta }\sigma(x,-H)\vert\partial_z(\vartheta_\delta+h_{u_\delta,\delta})\vert^2\,\rd (x,z)\,.
\end{split}
\end{equation*}
Since  $\vartheta_\delta(\cdot,-H-\delta)= 0$ a.e. in $D$, we infer from H\"older's inequality that
$$
\big\vert (\vartheta_\delta+h_{u_\delta,\delta})(x,-H)-h_{u_\delta,\delta}(x,-H-\delta)\big\vert^2\le \delta \int_{-H-\delta}^{-H}\vert\partial_z(\vartheta_\delta +h_{u_\delta,\delta})(x,z)\vert^2\,\rd z
$$
for a.e. $x\in D$. Combining the previous two estimates and using (see \eqref{bobbybrown40} and \eqref{h000}-\eqref{h0})
$$
h_{u_\delta,\delta}(x,-H)=h_{u_\delta}(x,-H)\,,\quad h_{u_\delta,\delta}(x,-H-\delta)  = \mathfrak{h}_{u_\delta}(x)\,, \qquad x\in D\,,
$$
we deduce from \eqref{i}, \eqref{xiiia}, and Lemma~\ref{LL0} that
\begin{equation}\label{xv}
\begin{split}
\liminf_{\delta\rightarrow 0} \,\frac{1}{2} \int_{\mathcal{R}_\delta } & \sigma_\delta(x,z)\vert\nabla(\vartheta_\delta+h_{u_\delta,\delta})\vert^2\,\rd (x,z)\\
&\ge \frac{1}{2}\int_{D}\sigma(x,-H)\big\vert\vartheta_0(x,-H)+h_u(x,-H)-\mathfrak{h}_u(x)\big\vert^2\,\rd x\,.
\end{split}
\end{equation}
Noticing finally that
\begin{equation*}
\begin{split}
\liminf_{\delta\rightarrow 0} G_\delta[\vartheta_\delta]&= 
\liminf_{\delta\rightarrow 0} \,\dfrac{1}{2}\int_{\Omega_\delta(u_\delta)} \sigma_\delta \vert\nabla (\vartheta_\delta+h_{u_\delta,\delta})\vert^2\,\rd (x,z)\\
&\ge \liminf_{\delta\rightarrow 0} \, \dfrac{1}{2}\int_{\mathcal{R}_\delta} \sigma_\delta \vert\nabla (\vartheta_\delta+h_{u_\delta,\delta})\vert^2\,\rd (x,z)\\
&\quad 
+ \liminf_{\delta\rightarrow 0} \bigg\{
\dfrac{1}{2}\displaystyle\int_{\Omega(u_\delta)}  \vert\nabla (\vartheta_\delta+h_{u_\delta,\delta})\vert^2\,\rd (x,z) \\
&\qquad\qquad\qquad\qquad  + \frac{1}{2}\int_{D}\big(\sigma\big\vert\vartheta_\delta+h_{u_\delta}-\mathfrak{h}_{u_\delta}\big\vert^2\big)(x,-H)\,\rd x\bigg\}\\
& \quad - \lim_{\delta\rightarrow 0}\frac{1}{2}\int_{D}\big(\sigma\big\vert\vartheta_\delta+h_{u_\delta}-\mathfrak{h}_{u_\delta}\big\vert^2\big)(x,-H)\,\rd x\,,
\end{split}
\end{equation*}
we readily obtain from \eqref{xi}, \eqref{xii}, and \eqref{xv} that
\begin{equation*}
\begin{split}
\liminf_{\delta\rightarrow 0} G_\delta[\vartheta_\delta]&\ge \dfrac{1}{2}\int_{\Omega(u)} \vert\nabla (\vartheta_0+h_{u})\vert^2\,\rd (x,z)+   \frac{1}{2}\int_{D}\big(\sigma\big\vert\vartheta_0+h_{u}-\mathfrak{h}_u\big\vert^2\big)(x,-H)\,\rd x\\
& = G[\vartheta_0]\,.
\end{split}
\end{equation*}
This is the asymptotic weak lower semicontinuity \eqref{AWLS}.

\medskip

\noindent\textit{(ii) Recovery sequence.}  Let $\hat{\Omega}(M):=D\times(-2H-M,M)$. Given an arbitrary function  $\vartheta\in H_B^1(\Omega(u))$ we define $\bar\vartheta\in H_0^1(\hat\Omega(M))$ by extending $\vartheta$ trivially to $D\times (-H,M)$ and then reflecting the outcome to $\hat{\Omega}(M)$; that is,
$$
\bar\vartheta(x,z):= \left\{ \begin{array}{ll} 0\,, & x\in D\,,\ u(x)<z<M\,, \\[0.1cm]
\vartheta(x,z)\,, &  x\in D\,,\ -H<z\le u(x)\,,\\[0.1cm]
\vartheta(x,-2H-z)\,, &  x\in D\,,\ -2H-u(x)<z\le -H\,,\\[0.1cm]
0\,, &  x\in D\,,\ -2H-M<z\le -2H-u(x)\,.\\
\end{array} \right.
$$
Then $F:=-\Delta \bar\vartheta\in H^{-1}(\hat\Omega(M))$. With
\begin{equation*}
\hat\Omega(u_\delta):=\Omega(u_\delta) \cup \big(D\times (-2H-M,-H]\big)\subset \hat\Omega(M)\,,
\end{equation*} 
the restriction of the distribution $F$ belongs to $H^{-1}(\hat\Omega(u_\delta))$. Thus, there is a unique variational solution $\hat\vartheta_\delta\in H_0^1(\hat\Omega(u_\delta))\subset H_0^1(\hat\Omega(M))$ to
$$
-\Delta \hat\vartheta_\delta=F \quad\text{in }\ \hat\Omega(u_\delta)\,,\qquad \hat\vartheta_\delta=0 \quad\text{on }\ \partial\hat\Omega(u_\delta)\,.
$$
If $d_H$ denotes  the Hausdorff distance in $\hat\Omega(M)$ (see \cite[Section~2.2.3]{HP05}), then, due to \eqref{i}  and the continuous embedding of $H_0^1(D)$ in $C(\bar{D})$, we have
$$
d_H\left( \hat\Omega(u_\delta),\hat\Omega(u) \right) \le \|u_\delta - u\|_{L_\infty(D)}\rightarrow 0\,.
$$
Since $\overline{\hat\Omega(M)}\setminus\hat\Omega(u_\delta)$ has a single connected component, it follows from \cite[Theorem~4.1]{Sv93} and \cite[Theorem~3.2.5]{HP05} that $\hat\vartheta_\delta\rightarrow \hat\vartheta$ in $H_0^1(\hat\Omega(M))$, where $\hat\vartheta\in H_0^1(\hat\Omega(M))$ is the unique variational solution to
$$
-\Delta \hat\vartheta=F  = -\Delta \bar{\vartheta} \quad\text{in }\ \hat\Omega(M)\,,\qquad \hat\vartheta=0 \quad\text{on }\ \partial\hat\Omega(M)\,.
$$
Clearly, since $\bar{\vartheta}$ and $\hat{\vartheta}$ both belong to $H_0^1(\hat\Omega(M))$, we deduce from the above identity that $\hat\vartheta=\bar\vartheta$. Hence,  
\begin{equation}\label{xvi}
\hat{\vartheta}_\delta  \rightarrow \bar\vartheta\quad \text{ in }\ H_0^1(\hat{\Omega}(M))\,.
\end{equation}
Considering the corresponding restrictions to $\Omega(M)$ yields 
\begin{equation}\label{xvii}
\hat\vartheta_\delta\rightarrow \bar\vartheta\quad \text{ in }\ H^1(\Omega(M))\,.
\end{equation}
Set
$$
\tau_\delta(x):= \left\{ \begin{array}{ll} 1\,, & L-\vert x\vert>\sqrt{\delta}\,,\\ \vspace{-3mm} \\[0.05cm]
\displaystyle\frac{L-\vert x\vert}{\sqrt{\delta}}\,, & L-\vert x\vert\le\sqrt{\delta}\,, \end{array} \right.\qquad x\in D\, ,
$$
and introduce
\begin{equation*}
\begin{split}
\vartheta_\delta(x,z):=&\ \frac{z+H+\delta}{\delta}\hat\vartheta_\delta(x,z) +\frac{z+H+\delta}{\delta}\big[ h_{u_\delta,\delta}(x,-H)-h_{u_\delta,\delta}(x,-H-\delta)\big] \tau_\delta (x)\\
&-\big[ h_{u_\delta,\delta}(x,z)-h_{u_\delta,\delta}(x,-H-\delta)\big]  \tau_\delta (x)\,,\quad (x,z)\in \mathcal{R}_\delta \,,
\end{split}
\end{equation*}
and
$$
\vartheta_\delta(x,z):=\hat\vartheta_\delta(x,z)\,,\quad (x,z)\in \Omega(u_\delta)\,.
$$
The smoothness and definitions of $\hat\vartheta_\delta$,  $h_{u_\delta,\delta}$, and $\tau_\delta$ imply that $\vartheta_\delta\in H^1(\mathcal{R}_\delta )\cap H^1(\Omega(M))$ and thus, since moreover $\llbracket \vartheta_\delta\rrbracket=0$ on $\Sigma(u_\delta)$, we deduce that $\vartheta_\delta\in H^1(\Omega_\delta(u_\delta))$. By construction, $\vartheta_\delta$ vanishes on $\partial \Omega_\delta(u_\delta)$, hence $\vartheta_\delta\in H_0^1(\Omega_\delta(u_\delta))$. We now claim that $(\vartheta_\delta)_{\delta\in (0,1)}$ is a recovery sequence for $\vartheta$; that is,
\begin{equation}\label{RS}
G[\vartheta]=\lim_{\delta\rightarrow 0} G_\delta[\vartheta_\delta]\,.
\end{equation}
First, using that $\hat\vartheta_\delta=0$ in $\Omega(M)\setminus\Omega(u_\delta)$ and $\vartheta_\delta=\hat\vartheta_\delta$ in $\Omega(u_\delta)$ along with \eqref{i}, Lemma ~\ref{LL0}, and \eqref{xvii}, it is not difficult to see that
\begin{equation}\label{xviii}
\lim_{\delta\to 0}\, \frac{1}{2}\int_{\Omega(u_\delta)}\vert \nabla(\vartheta_\delta +h_{u_\delta,\delta})\vert^2\,\rd (x,z)=\frac{1}{2}\int_{\Omega(u)}\vert \nabla(\vartheta +h_{u})\vert^2\,\rd (x,z)\,.
\end{equation}
Next, for $(x,z)\in \mathcal{R}_\delta$, we have
\begin{equation}\label{xx}
\begin{split}
\partial_z(\vartheta_\delta+h_{u_\delta,\delta})(x,z)&=\frac{1}{\delta}\hat\vartheta_\delta(x,z) + \frac{1}{\delta}\big[ h_{u_\delta,\delta}(x,-H)-h_{u_\delta,\delta}(x,-H-\delta)\big] \tau_\delta (x) \\
& \qquad +\frac{z+H+\delta}{\delta}\partial_z\hat\vartheta_\delta(x,z)+\big(1-\tau_\delta (x) \big)\partial_z h_{u_\delta,\delta}(x,z)\,,
\end{split}
\end{equation}
 and we aim at identifying the limit of the right-hand side of \eqref{xx} as $\delta\rightarrow 0$. Let us first note that, for $z\in (-H-\delta,-H)$,
\begin{align*}
\int_D \left| \hat{\vartheta}_\delta(x,z) - \hat{\vartheta}_\delta(x,-H) \right|^2\,\rd x & \le \int_D |H+z| \int_z^{-H} \left| \partial_z \hat{\vartheta}_\delta \right|^2\, \rd z\,\rd x \\
& \le \delta \int_{\mathcal{R}_\delta} \left| \nabla \hat{\vartheta}_\delta \right|^2\,\rd (x,z)\,,
\end{align*}
from which, thanks to the convergence \eqref{xvi}, we deduce that
\begin{equation}
\lim_{\delta\rightarrow 0} \frac{1}{\delta}\int_D \left| \hat{\vartheta}_\delta(x,z) - \hat{\vartheta}_\delta(x,-H) \right|^2\,\rd x = 0\,. \label{xp1}
\end{equation}
Since \eqref{xvii} implies that $\hat\vartheta_\delta(\cdot,-H)\to \bar\vartheta(\cdot,-H)$ in $L_2(D)$, we infer from \eqref{xp1} and the continuity of $\sigma$ that
\begin{equation}\label{xxii}
\lim_{\delta\to 0}\,\frac{1}{\delta}\int_{-H-\delta}^{-H}\int_D \sigma(x,z)\vert \hat\vartheta_\delta(x,z)\vert^2\,\rd x\,\rd z=\int_D\sigma(x,-H)\vert\bar\vartheta(x,-H)\vert^2\,\rd x\,.
\end{equation}
Now, the definitions of $\sigma_\delta=\delta\sigma$ in $\mathcal{R}_\delta $ and $\tau_\delta$, the properties of $h_{u_\delta,\delta}$ (see Lemma~\ref{LL0}), and \eqref{xxii} yield
\begin{align}\label{xxiii}
&\lim_{\delta\to 0}\, \int_{\mathcal{R}_\delta}\sigma_\delta(x,z) \left\vert \frac{1}{\delta}\hat\vartheta_\delta(x,z) + \frac{1}{\delta}\big[ h_{u_\delta,\delta}(x,-H)-h_{u_\delta,\delta}(x,-H-\delta)\big] \tau_\delta (x)\right\vert^2\, \rd (x,z)
\nonumber \\
& \quad = \lim_{\delta\to 0}\,  \frac{1}{\delta} \int_{-H-\delta}^{-H}\int_D\sigma(x,z) \left\vert \hat\vartheta_\delta(x,z) + \big[ h_{u_\delta,\delta}(x,-H)-h_{u_\delta,\delta}(x,-H-\delta)\big] \tau_\delta (x)\right\vert^2\,\rd x\,\rd z
\nonumber \\
&\quad  = \int_D\sigma(x,-H)\vert\bar\vartheta(x,-H)+h_u(x,-H)-\mathfrak{h}_u(x)\vert^2\,\rd x\,.
\end{align}
Moreover, 
\begin{align*}
\int_{\mathcal{R}_\delta }\sigma_\delta(x,z) \left\vert\frac{z+H+\delta}{\delta}\partial_z\hat\vartheta_\delta(x,z)\right\vert^2\,\rd (x,z) & \le
\delta \sigma_{max} \int_{-H-\delta}^{-H}\int_D \left\vert \partial_z\hat\vartheta_\delta(x,z)\right\vert^2\,\rd x\rd z\\
&\le
\delta \sigma_{max} \|\hat\vartheta_\delta\|_{H^1(\hat\Omega(M))}^2
\end{align*}
so that, recalling that $(\hat\vartheta_\delta)_{\delta\in (0,1)}$ is bounded in $H^1(\hat\Omega(M))$ due to \eqref{xvi},
\begin{align}\label{xxiv}
\lim_{\delta\rightarrow 0}\, \int_{\mathcal{R}_\delta }\sigma_\delta(x,z) \left\vert\frac{z+H+\delta}{\delta}\partial_z\hat\vartheta_\delta(x,z)\right\vert^2\,\rd (x,z)=0\,.
\end{align}
Finally, observe from \eqref{bobbybrown40} that
\begin{equation*}
\partial_z h_{u_\delta,\delta}(x,z)=\frac{1}{\delta}\partial_z h_{b}\left(x,-H+\frac{z+H}{\delta},u_\delta(x)\right)\,,\quad (x,z)\in \mathcal{R}_\delta \,.
\end{equation*}
Hence,
\begin{align*}
\int_{\mathcal{R}_\delta }\sigma_\delta(x,z) &\left\vert\big(1-\tau_\delta (x) \big)\partial_z h_{u_\delta,\delta}(x,z)\right\vert^2\,\rd (x,z)\nonumber\\
&\le \sigma_{max}\int_{-H-1}^{-H}\int_D \big\vert \big(1-\tau_\delta (x) \big) \partial_z h_{b}\left(x,\xi,u_\delta(x)\right)\big\vert^2\,\rd x\rd \xi
\end{align*}
so that, using \eqref{i}, the definition of $\tau_\delta$, the continuity of $\partial_z h_b$, and Lebesgue's dominated convergence theorem, we derive
\begin{align}\label{xxv}
\lim_{\delta\rightarrow 0}\, \int_{\mathcal{R}_\delta }\sigma_\delta(x,z) \left\vert\big(1-\tau_{\delta} (x) \big)\partial_z h_{u_\delta,\delta}(x,z)\right\vert^2\,\rd (x,z)=0\,.
\end{align}
Consequently, we deduce from \eqref{xx} and \eqref{xxiii}-\eqref{xxv} that
\begin{equation}\label{xxvi}
\begin{split}
\lim_{\delta\rightarrow 0}\, &\int_{\mathcal{R}_\delta }\sigma_\delta(x,z) \vert\partial_z(\vartheta_\delta+h_{u_\delta,\delta})\vert^2\,\rd (x,z)\\
&= \int_{D}\sigma(x,-H) \left\vert \bar\vartheta(x,-H) +  h_{u}(x,-H)-\mathfrak{h}_u(x) \right\vert^2\,\rd x \,.
\end{split}
\end{equation}
Furthermore, we note that
\begin{equation*}
\begin{split}
\partial_x\vartheta_\delta(x,z)=&\ \frac{z+H+\delta}{\delta}\partial_x\hat\vartheta_\delta(x,z)\\
& +\frac{z+H+\delta}{\delta}\big[ \partial_x h_{u_\delta,\delta}(x,-H)- \partial_x h_{u_\delta,\delta}(x,-H-\delta)\big] \tau_\delta (x)\\
&+\frac{z+H+\delta}{\delta}\big[ h_{u_\delta,\delta}(x,-H)-  h_{u_\delta,\delta}(x,-H-\delta)\big]  \partial_x\tau_\delta (x)\\
&-\big[ \partial_xh_{u_\delta,\delta}(x,z)-\partial_xh_{u_\delta,\delta}(x,-H-\delta)\big]  \tau_\delta (x)\\
&-\big[ h_{u_\delta,\delta}(x,z)-h_{u_\delta,\delta}(x,-H-\delta)\big]  \partial_x\tau_\delta (x)
\end{split}
\end{equation*}
and, recalling \eqref{bobbybrown40},
\begin{equation*}
\begin{split}
 \partial_x h_{u_\delta,\delta}(x,z) & = \partial_x h_{b} \left(x,-H+\frac{z+H}{\delta}, u_\delta(x)\right) \\
& \qquad +\partial_x u_\delta(x) \partial_wh_b\left(x,-H+\frac{z+H}{\delta}, u_\delta(x)\right)
\end{split}
\end{equation*}
for $(x,z)\in \mathcal{R}_\delta$. Thus, since
$$
0\le \tau_\delta(x)\le 1\,,\qquad 0\le \frac{z+H+\delta}{\delta}\le 1\,,\qquad (x,z)\in\mathcal{R}_\delta \,,
$$
we easily obtain from $\sigma_\delta=\delta\sigma$ in $\mathcal{R}_\delta$ that
\begin{equation*}
\begin{split}
\int_{\mathcal{R}_\delta }&\sigma_\delta(x,z)\big\vert\partial_x(\vartheta_\delta+h_{u_\delta,\delta})\big\vert^2\rd (x,z)\\
&\le \  c\, \delta \sigma_{max}\int_{\mathcal{R}_\delta } \vert\partial_x\hat\vartheta_\delta(x,z)\vert^2\, \rd (x,z)\\
&\qquad + c\, \delta^2 \sigma_{max} \| h_b\|_{C^1}^2 \int_D   \left(1+\vert\partial_x u_\delta(x)\vert^2+\vert \partial_x\tau_\delta(x)\vert^2\right) \,\rd x\\
&\le c\, \delta \sigma_{max}\|\hat\vartheta_\delta\|_{H^1(\hat\Omega(M))}^2 
+  c\, \delta^2 \sigma_{max} \| h_b\|_{C^1}^2 \left(\vert D\vert +\|u_\delta\|_{H^1(D)}^2+\frac{\vert D\vert}{\delta}\right)\,,
\end{split}
\end{equation*}
where $\|h_b\|_{C^1}$ denotes the norm of $h_b$ in $C^1(\bar D\times[-H-1,-H]\times [-H,M])$, and $c$ is a positive constant depending on $D$ and $H$. Therefore, \eqref{i} and \eqref{xvi} entail
\begin{equation}\label{xxvii}
\lim_{\delta\rightarrow 0}\, \int_{\mathcal{R}_\delta }\sigma_\delta(x,z)\big\vert\partial_x(\vartheta_\delta+h_{u_\delta,\delta})\big\vert^2\,\rd (x,z)=0\,.
\end{equation}
Consequently, we derive from \eqref{xviii}, \eqref{xxvi}, and \eqref{xxvii} that
\begin{equation*}
\begin{split}
\lim_{\delta\rightarrow 0} G_\delta[\vartheta_\delta]&=
\lim_{\delta\rightarrow 0} \bigg(  \frac{1}{2}\int_{\Omega(u_\delta)}  \vert\nabla (\vartheta_\delta+h_{u_\delta,\delta})\vert^2\,\rd (x,z)\\
&\qquad\qquad+\frac{1}{2}\int_{\mathcal{R}_\delta} \sigma_\delta \big(\vert\partial_x(\vartheta_\delta+h_{u_\delta,\delta})\vert^2+\vert\partial_z(\vartheta_\delta+h_{u_\delta,\delta})\vert^2\big)\,\rd (x,z)
\bigg)
\\
&= \frac{1}{2}\int_{\Omega(u)}  \big\vert\nabla(\vartheta+h_u)\big\vert^2\,\rd (x,z)\\
&\qquad  +
\frac{1}{2}\int_{D}\sigma(x,-H) \left\vert \bar\vartheta(x,-H) +  h_{u}(x,-H)-\mathfrak{h}_{u}(x) \right\vert^2\,\rd x \\
& =  G[\vartheta]\,,
\end{split}
\end{equation*}
where we used that $ \bar\vartheta(x,-H)=\vartheta(x,-H)$ by construction of $ \bar\vartheta$. Hence, $(\vartheta_\delta)_{\delta\in (0,1)}$ is indeed a recovery sequence for $\vartheta$. 

\medskip

\noindent \textit{(iii) Convergence.} Since \textit{(i)} and \textit{(ii)} prove \eqref{GG}, we may invoke the Fundamental Theorem of $\Gamma$-convergence \cite[Corollary~7.20]{DaM93} to deduce from \eqref{G1}-\eqref{GG} that, as $\delta\to 0$,
$$
E_{e,\delta}(u_\delta)=-G_\delta[\chi_{u_\delta,\delta}]\longrightarrow -G[\chi_{u}]=E_{e,0}(u)
$$
and 
$$
\psi_{u_\delta,\delta}-h_{u_\delta,\delta} \longrightarrow \psi_{u}-h_{u} \quad\text{in }\ L_2(\Omega(M))\,.
$$
This proves Proposition~\ref{PP1}.
\end{proof}

\subsection{$\Gamma$-convergence of the total energy}\label{S3.2}

We now turn to the $\Gamma$-convergence of the total energy and first establish that the $H^2$-norm of $u$ is controlled by the total energy $E_\delta(u)$ (defined in \eqref{te}) and the $L_2$-norm of $u$, whatever the value of $\delta\in (0,1)$.

\begin{lemma}\label{LL1}
Given $\kappa>0$ there is  a constant $c(\kappa)>0$ such that, if $u\in\bar{S}_0$ satisfies 
\begin{equation}\label{A1}
\| u\|_{L_2(D)}\le \kappa\qquad\text{ and }\qquad E_\delta(u)\le \kappa\,,\quad \delta\in (0,1)\,,
\end{equation}
then
\begin{equation}\label{H2}
\| u\|_{H^2(D)} + \int_{\Omega_\delta(u)} \sigma_\delta \vert \nabla \psi_{u,\delta}\vert^2\, \rd (x,z)\le c(\kappa)\,,\quad \delta\in (0,1)\,.
\end{equation}
\end{lemma}

\begin{proof}
We argue similarly to \cite[Lemma 2.3]{JEPE20}. The variational characterization of $\psi_{u,\delta}$ (see \cite[Lemma~3.2]{LW19}) and \eqref{H1} imply
\begin{equation}\label{pl}
\begin{split}
\int_{\Omega_\delta(u)} \sigma_\delta \vert \nabla \psi_{u,\delta}\vert^2\, \rd (x,z) & \le \int_{\Omega_\delta(u)} \sigma_\delta \vert \nabla h_{u,\delta}\vert^2\, \rd (x,z) \\ &\le  c_0\big(1+\|u\|_{L_2(D)}^2+ \| \partial_x u\|_{L_2(D)}^2\big)\,,
\end{split}
\end{equation}
where $c_0$ is defined in Lemma~\ref{LL0}. 
Furthermore, since $u\in \bar{S}_0\subset H_D^2(D)$ we have
\begin{equation}\label{e21}
\|\partial_x u\|_{L_2(D)}^2=-\int_D u\partial_x^2u\,\rd x\le \|u\|_{L_2(D)} \| \partial_x^2 u\|_{L_2(D)}\,,
\end{equation}
so that we deduce from \eqref{A1} and \eqref{pl} that
\begin{equation}\label{h3}
\begin{split}
-E_{e,\delta}(u)=\frac{1}{2}\int_{\Omega_\delta(u)} \sigma_\delta \vert \nabla \psi_{u,\delta}\vert^2\, \rd (x,z)
&\le c(\kappa) \big(1+ \| \partial_x^2 u\|_{L_2(D)}\big)\,.
\end{split}
\end{equation}
Consequently, we obtain from \eqref{h3}, the definition of $E_\delta$, and Young's inequality that
\begin{equation*}
\begin{split}
E_\delta (u)&\ge \frac{\beta}{2} \|\partial_x^2u\|_{L_2(D)}^2 -c(\kappa) \big(1+ \| \partial_x^2 u\|_{L_2(D)}\big)\ge  \frac{\beta}{4} \|\partial_x^2u\|_{L_2(D)}^2 -c(\kappa)\,.
\end{split}
\end{equation*}
Combining the above estimate with \eqref{A1} and \eqref{e21} entails that $\| u\|_{H^2(D)}\le c(\kappa)$, which also implies the second assertion of \eqref{H2} due to \eqref{h3}.
\end{proof}

The total energies (defined in \eqref{te} and \eqref{rte}), being \textit{a priori} defined  only on $\bar{S}_0$, are extended to functionals on $L_2(D)$ by setting
$$
E_\delta(u):=\infty\,,\quad E(u):=\infty\,,\qquad u\in L_2(D)\setminus \bar{S}_0\,.
$$
Then we can prove:

\begin{corollary}\label{PA}
\begin{equation*}
\Gamma-\lim_{\delta\rightarrow 0} E_\delta =E\quad\text{in }\ L_2(D)\,.
\end{equation*}
\end{corollary}

\begin{proof}
\textit{(i) Recovery sequence.} Concerning the construction of a recovery sequence it is sufficient to consider $u\in \bar{S}_0$. Let us observe from \cite[Corollary~3.4]{AMOP20}  that
$$
\lim_{\delta\to 0} E_{e,\delta}(u)=E_{e,0}(u)\,.
$$
Since $E_m(u)$ is independent of $\delta$, we thus readily obtain
$$
\lim_{\delta\to 0} E_{\delta}(u)=E(u)\,.
$$

\noindent{\textit{(ii) Asymptotic weak lower semicontinuity.}} Consider a sequence $(u_\delta)_{\delta\in (0,1)}$ in~$L_2(D)$ and $u\in L_2(D)$ such that
\begin{equation}\label{17}
\lim_{\delta\to 0} \|u_\delta-u\|_{L_2(D)}=0\,.
\end{equation}
Since we shall show that then
\begin{equation}\label{18}
E(u)\le\liminf_{\delta\to 0} E_\delta(u_\delta)\,,
\end{equation}
a property which is obviously true if the right-hand side is infinite, we may assume that there is a constant $\kappa>0$ such that
\begin{equation}\label{19}
E_\delta(u_\delta)\le \kappa\,,\quad \delta\in (0,1)\,.
\end{equation}
Now, due to \eqref{17} and \eqref{19}, we may invoke Lemma~\ref{LL1} to derive that $(u_\delta)_{\delta\in (0,1)}$ is bounded in $H^2(D)$. Thus, up to a subsequence, we have $u_\delta \rightharpoonup u$ in $H^2(D)$ and $u_\delta \to u$ in~$H^1(D)$. The former implies
\begin{equation}\label{24}
E_m(u)\le\liminf_{\delta\to 0} E_m(u_\delta)\,,
\end{equation}
while the latter, along with Proposition~\ref{PP1}, entails
\begin{equation}\label{25}
\lim_{\delta\to 0} E_{e,\delta}(u_\delta)=E_{e,0} (u)\,.
\end{equation}
Therefore, \eqref{18} holds true owing to \eqref{24} and \eqref{25}.
This implies the assertion.
\end{proof} 

\subsection{Remaining arguments for the proof of Theorem~\ref{TT1}:  The case $\bm{a>0}$}\label{S3.3}

Let $\delta\in (0,1)$. We first use the positivity of $a$ to show that the $H^2$-norm is controlled by $E_\delta$. Specifically, it follows from \eqref{pl}, the  Poincar\'e inequality
\begin{equation}
	\|v\|_{L_2(D)} \le 4L \|\partial_x v\|_{L_2(D)}\,, \qquad v\in H_0^1(D)\,, \label{pi}
\end{equation}
and Young's inequality $a r^4 + a \ge 2 a r^2$ that, for $u\in \bar{S}_0$, 
\begin{align*}
	E_\delta(u) & \ge \frac{\beta}{2} \|\partial_x^2 u\|_{L_2(D)}^2 + \frac{a}{4} \|\partial_x u\|_{L_2(D)}^4 - c_0 \left( 1 + \|u\|_{L_2(D)}^2 + \|\partial_x u\|_{L_2(D)}^2 \right) \\
	& \ge \frac{\beta}{2} \|\partial_x^2 u\|_{L_2(D)}^2 + \frac{a}{4} \|\partial_x u\|_{L_2(D)}^4 - c_0 \left[ 1 + (1+16L^2) \|\partial_x u\|_{L_2(D)}^2 \right] \\
	& \ge \frac{\beta}{2} \|\partial_x^2 u\|_{L_2(D)}^2 + \frac{a}{8} \|\partial_x u\|_{L_2(D)}^4 - c_0 - \frac{c_0^2}{a} (1+16L^2)^2 \\
	& \ge \frac{\beta}{2} \|\partial_x^2 u\|_{L_2(D)}^2 + \frac{a}{4} \|\partial_x u\|_{L_2(D)}^2 - c_1\,,
\end{align*}
with $c_0$ defined in Lemma~\ref{LL0} and $c_1 := a/8 + c_0 + c_0^2 (1+16L^2)^2/a$. Hence,
\begin{equation}
	\frac{\beta}{2} \|\partial_x^2 u\|_{L_2(D)}^2 + \frac{a}{4} \|\partial_x u\|_{L_2(D)}^2 \le E_\delta(u) + c_1\,, \qquad u\in \bar{S}_0\,. \label{vs}
\end{equation}

Now, for each $\delta\in (0,1)$, let $u_\delta^*\in \bar{S}_0$ be an arbitrary minimizer of $E_\delta$ in $\bar{S}_0$, see \eqref{B}, with corresponding electrostatic potential $\psi_{u_\delta^*,\delta}$ satisfying \eqref{TMP}. Since $E_\delta(u_\delta^*) \le E_\delta(0) \le 0$, we readily infer from \eqref{pi} and \eqref{vs} that $(u_\delta^*)_{\delta\in (0,1)}$ is bounded in $H^2(D)$. In particular, there are a subsequence $\delta_j\rightarrow 0$ and  $u^*\in \bar{S}_0$  such that
\begin{equation}
u_{\delta_j}^*\rightharpoonup u^*\quad\text{ in }\ H^2(D)\,, \label{wch2}
\end{equation}
so that Corollary~\ref{PA} and the Fundamental Theorem of $\Gamma$-convergence, see \cite[Corollary~7.20]{DaM93}, imply that $u^*$ is a minimizer of $E$ on $\bar{S}_0$ and
\begin{equation}
\lim_{j\to\infty} E_{\delta_j}(u_{\delta_j}^*) = E(u^*)\,. \label{cve}
\end{equation}
Moreover, since $(u_\delta^*)_{\delta\in (0,1)}$ is bounded in $H^2(D)$ and $(E_\delta(u_\delta^*))_{\delta\in (0,1)}$ is bounded, Lemma~\ref{LL1} and \eqref{sigma} entail that (the trivial extensions of)  $\big(\psi_{u_\delta^*,\delta}-h_{u_\delta^*,\delta}\big)_{\delta\in (0,1)}$ is bounded in $H^1(D\times (-H,M))$, where 
\begin{equation*}
	M := \max\left\{ \|H+u^*\|_{L_\infty(D)} , \sup_{\delta\in (0,1)} \|H+u_{\delta}^*\|_{L_\infty(D)}   \right \}
\end{equation*} 
is finite thanks to the boundedness of $(u_\delta^*)_{\delta\in (0,1)}$ in $H^2(D)$ and the continuous embedding of $H^2(D)$ in $L_\infty(D)$. Therefore, upon extracting a further subsequence if necessary, we may assume that $\big( \psi_{u_{\delta_j}^*,\delta_j} - h_{u_{\delta_j}^*,\delta_j} \big)_{j\ge 1}$ weakly converges in $H^1(D\times (-H,M))$, the limit necessarily being $\psi_{u^*}-h_{u^*}$ owing to Proposition~\ref{PP1}. 

 Let us finally improve the convergence \eqref{wch2} of $(u_{\delta_j}^*)_{j\ge 1}$. Since $H^2(D)$ embeds compactly in $H^1(D)$, it follows from \eqref{wch2} that 
\begin{equation}
\label{w2}
u_{\delta_j}^*\rightarrow u^*\quad\text{ in }\ H^1(D)
\end{equation}
and Proposition~\ref{PP1} then entails that $E_{e,\delta_j}(u_{\delta_j}^*)\rightarrow E_{e,0}(u^*)$ as $j \rightarrow \infty$. Recalling \eqref{cve}, we deduce that $E_{m}(u_{\delta_j}^*) \rightarrow E_{m}(u^*)$ as $j \rightarrow \infty$. Together with the convergences \eqref{wch2} and \eqref{w2}, this property implies the strong convergence of $(u_{\delta_j}^*)_{j\geq 1}$ to $u^*$ in $H^2(D)$ and completes the proof of Theorem~\ref{TT1} when $a>0$.

\subsection{Remaining arguments for the proof of Theorem~\ref{TT1}: The case $\bm{a=0}$}\label{S3.4}

To finish off the proof of Theorem~\ref{TT1}, we are left with the case $a=0$ for which the weak compactness of minimizers in $H^2(D)$ is harder to derive. Additional information on these minimizers is actually required and follows from the analysis performed in \cite{LW19, JEPE20}, using that they are critical points of the total energy.
	
\begin{lemma}\label{LLzz}
There is a constant $c_2>0$ which does not depend on $\delta\in (0,1)$ such that, if  $u$ is a minimizer of $E_\delta$ on $\bar{S}_0$ for some $\delta\in (0,1)$, then
\begin{equation*}
\|u\|_{L_\infty(D)}\le c_2\,,\qquad \delta\in (0,1)\,.
\end{equation*}
\end{lemma}

Taking Lemma~\ref{LLzz} for granted, we are in a position to complete the proof of Theorem~\ref{TT1} when $a=0$. 

\begin{proof}[Proof of Theorem~\ref{TT1}: $a=0$]
For each $\delta\in (0,1)$, let $u_\delta^*\in \bar{S}_0$ be an arbitrary minimizer of $E_\delta$ in $\bar{S}_0$, see \eqref{B}, with corresponding electrostatic potential $\psi_{u_\delta^*,\delta}$ satisfying \eqref{TMP}. 	
By Lemma~\ref{LLzz}, $(u_\delta^*)_{\delta\in (0,1)}$ is bounded in $L_\infty(D)$ and thus also in $L_2(D)$. Therefore, since $E_\delta(u_\delta^*)\le E_\delta(0)\le 0$, it is also bounded in $H^2(D)$ according to Lemma~\ref{LL1}. We may then proceed as in the previous case $a>0$ in order to complete the proof of Theorem~\ref{TT1}.
\end{proof}

We are left with proving Lemma~\ref{LLzz}, which relies on the same comparison argument as \cite[Proposition~2.6]{JEPE20} and uses in an essential way the Euler-Lagrange equation satisfied by minimizers of the total energy $E_\delta$.

\begin{proof}[Proof of Lemma~\ref{LLzz}]
Let $\delta\in (0,1)$ and consider a minimizer $u\in \bar{S}_0$ of $E_\delta$ on $\bar{S}_0$ (if any). Owing to \eqref{bobbybrown}, it follows from \cite[Theorem~1.3]{JEPE20} (see also \cite[Theorem~5.3]{LW19}) that $u$ is a weak solution to the parabolic variational inequality
\begin{equation*}	
\beta\partial_x^4 u - \tau\partial_x^2 u + \partial\mathbb{I}_{\bar{S}_0}(u) \ni -g_\delta(u) \;\;\text{ in }\;\; D\,,
\end{equation*}
where $\partial\mathbb{I}_{\bar{S}_0}$ denotes the subdifferential in $L_2(D)$ of the  indicator function $\mathbb{I}_{\bar{S}_0}$ of the closed convex set $\bar{S}_0$ (that is, $\mathbb{I}_{\bar{S}_0}(v)=0$ for $v\in \bar{S}_0$ and $\mathbb{I}_{\bar{S}_0}(v)=\infty$ for $v\in L_2(D)\setminus \bar{S}_0$).  Taking into account assumptions \eqref{sigma} and \eqref{Kbound0}, the electrostatic force $g_\delta(u)\in L_2(D)$ is given by
\begin{subequations} \label{ef}
\begin{equation}
g_\delta(u)(x) := \mathfrak{g}_\delta(u)(x) -\frac{1}{2}  \left[ \big((\partial_x h)_u\big)^2+ \big((\partial_z h)_u+(\partial_w h)_u\big)^2 \right](x, u(x)) \label{efa}
\end{equation}
for $x\in D$, where 
\begin{equation}
\mathfrak{g}_\delta(u)(x) := \frac{1}{2} \big(1+(\partial_x u(x))^2\big)\,\big[\partial_z\psi_{u,\delta}^2-(\partial_z h)_u-(\partial_w h)_u\big]^2(x, u(x)) \label{efb}
\end{equation}
for $x\in D\setminus \mathcal{C}(u)$ and
\begin{equation}
\mathfrak{g}_\delta(u)(x) := \frac{1}{2} \left[ \sigma_\delta\partial_z\psi_{u,\delta,1} - (\partial_z h)_u-(\partial_w h)_u \right]^2(x, -H) \label{efc}
\end{equation}
\end{subequations}
for $x\in  \mathcal{C}(u)$, the coincidence set $\mathcal{C}(u)$ being defined in \eqref{cs}. In the definition of $\mathfrak{g}_\delta(u)$, $\psi_{u,\delta,1} := \psi_{u,\delta} \mathbf{1}_{\mathcal{R}_\delta}$ and $\psi_{u,\delta,2} := \psi_{u,\delta} \mathbf{1}_{\Omega(u)}$, where we recall that \cite[Theorem~1.1]{LW19} guarantees that $\psi_{u,\delta,1} \in H^2(\mathcal{R}_\delta)$ and $\psi_{u,\delta,2}\in H^2(\Omega(u))$, so that the traces involved in \eqref{ef} are well-defined.

Now, since $\mathfrak{g}_\delta(u_\delta^*)\ge 0$ in $D$,  it easily follows from \eqref{Kbound} that $g_\delta(u)\ge  -K^2$ in $D$ and we argue as in the proof of \cite[Proposition~2.6]{JEPE20} to conclude that there is a constant $c>0$ depending only on $L$, $\beta$, $\tau$, and $K$ such that $u\le c$ in $D$. Recalling that $u\ge -H$ completes the proof.
\end{proof}

\section*{Acknowledgments}
We thank the referee for helpful comments.

\bibliographystyle{siam}
\bibliography{DL}

\end{document}